\documentclass[12pt]{article}
\usepackage{amsmath,amssymb,amsfonts,amsthm}
\usepackage{multirow}
\topmargin- 0.5cm
\newcounter{item}[section]
\newcounter{kirshr}
\newcounter{kirsha}
\newcounter{kirshb}
\newenvironment{mysect}[1]{\vskip8pt\par\noindent\setcounter{item}{1}
\setcounter{equation}{0}{\large\bf\arabic{section}.  #1 }\vskip8pt\nopagebreak\par\nopagebreak }
{\stepcounter{section}\upshape\par}

\newtheorem{theorem}{Theorem}[section]

\newtheorem{corollary}[theorem]{Corollary}
\newtheorem{proposition}[theorem]{Proposition}
\renewcommand{\b}{\nabla}

\newcommand{\oversetc}[1]{\overset{\text{\tiny}{\text{\tiny}\circ}}#1}
\newcommand\undersym[2]{\raisebox{-6pt}{\tiny$#2$}{\kern-5pt}\mbox{$#1$}}
\newcommand\overcirc[1]{\raisebox{10pt}{\tiny$\circ$}{\kern-7pt}\mbox{$#1$}}

\newtheorem{remark}[theorem]{Remark}

\newtheorem{definition}[theorem]{Definition}
\begin{document}
\title{ \textbf{Connections in sub-Riemannian geometry of parallelizable distributions}\footnote{arXiv number: 1603.06106 [math.DG]} }
\author{Nabil L. Youssef and Ebtsam H. Taha}
\date{}
%\thanks{\it Department of Mathematics, etc}
%\pagestyle{fancy}

\maketitle                     % Produces the title.
\vspace{-1.2cm}
\begin{center}
{Department of Mathematics, Faculty of Science,\\ Cairo University, Giza, Egypt} \end{center}
\vspace{-0.6cm}
\begin{center}
nlyoussef@sci.cu.edu.eg,\, nlyoussef2003@yahoo.fr \\
ebtsam.taha@sci.cu.edu.eg,\, ebtsam.h.taha@hotmail.com
\end{center}
\vspace{0.1cm}
\maketitle
\stepcounter{section}
\begin{flushright}
\emph{ Dedicated to the meomery of Waleed A. Elsayed}\\
\end{flushright}
\vspace{0.3cm}
%%%%%%%%%%%%%%%%%%%%%%%%%%%%%%%%%%%%%%%%%%%%%%%%%%%%%%% Abstract %%%%%%%%%%%%%%%%%%%%%%%%%%%%%%%%%%%%%%%%%%%%%%%%%%%%%%%%%%%%%%%%%%

\noindent{\bf Abstract.} The notion of a parallelizable distribution has been introduced and investigated. A non-integrable parallelizable distribution carries a natural sub-Riemannian structure. The geometry of this structure has been studied from the bi-viewpoint of absolute parallelism geometry and sub-Riemannian geometry. Two remarkable linear connections have been constructed on a sub-Riemannian parallelizable distribution, namely, the Weitzenb\"ock connection and the sub-Riemannian connection. The obtained results have been applied to two concrete examples: the spheres $S^3$ and $S^7$.

\medskip
\noindent{\bf Keywords:} bracket generating distribution, parallelizable distribution, sub-Riemannian structure, Weitzenb\"ock connection, sub-Riemannian connection, spheres $S^3$ and $S^7$.
%\vspace{0.5cm}

\vspace{0.3cm}

%%%%%%%%%%%%%%%%%%%%%%%%%%%%%%%%%%%%%%%%%%%%%%%%%%%%%%%%%%%%% Sec 1. %%%%%%%%%%%%%%%%%%%%%%%%%%%%%%%%%%%%%%%%%%%%%%%%%%%%%%%%%%

\begin{mysect}{Introduction}
Sub-Riemannian Geometry \cite{Agrachev, {Calin}, {Donne}, {Strichartz}, {Strichartz C}} has many applications such as diffusion, mechanics, gauge theories and control theory \cite{Frédéric, Rifford}. Absolute parallelism geometry or the geometry of parallelizable manifolds \cite{Brickell, Wanas, Waleed, AMR} is frequently used for applications in physics, especially in the geometrization of physical theories such general relativity and gravitation \cite{WN, NW, Shirafuji, charge, classical}.

Several attempts \cite{Bejancu, Cole, Robert, Hladky & Pauls} have been made to construct a connection theory in sub-Riemannian geometry. Our approach is different. We define a parallelizable distribution (PD) on a finite dimensional manifold $M$. A non-integrable PD on $M$ carries simultaneously two structures: an absolute parallelism structure and a sub-Riemannian structure. We make use of both structures to build up a connection theory on PD's. Two remarkable connections have been constructed on a sub-Riemannian parallelizable distribution, namely, the Weitzenb\"ock connection and the sub-Riemannian connection. The obtained results have been applied to the spheres $S^3$ and $S^7$.

 The wide spectrum of applications of both sub-Riemannian geometry and absolute parallelism geometry makes our approach, which enjoys the advantages of both geometries, a potential candidate for more applications in different fields.

\end{mysect}
\vspace{0.5cm}

%%%%%%%%%%%%%%%%%%%%%%%%%%%%%%%%%%%%%%%%%%%%%%%%%%%%%%%%%% Sec. 2 %%%%%%%%%%%%%%%%%%%%%%%%%%%%%%%%%%%%%%%%%%%%%
\begin{mysect}{Parallelizable distribution and Sub-Riemannian structure}

We first give some fundamental definitions concerning sub-Riemannian geometry. For a more detailed exposition of Sub-Riemannian geometry, we refer to \cite{Agrachev, {Calin}, {Donne}, {Strichartz}, {Strichartz C}}. By a manifold $M$, we mean an n-dimensional smooth manifold.

\begin{definition}
A distribution of rank $k$ on a manifold $M$ is a map $D$ that assigns to each point $p\in M$ a $k$-dimensional subspace $D_{p}$ of $T_pM$. A distribution $D$ may be regarded as a vector sub-bundle\, $(\bigcup_{p\in M}D_p)\longrightarrow M$\, of the tangent bundle $TM\longrightarrow M$.
A Distribution D of rank $k$ is differentiable if every point $p\in M$ has a neighborhood $U$ and smooth $k$-vector fields $Y_1,\cdots Y_k$ on $U$ such that $Y_1(q),\cdots Y_k(q)$ form a basis of $D_{q}$ for all $q \in U$, i.e., $ D=\text{Span}\{Y_{1}, \cdots Y_{k}\}$ on $U$.
\end{definition}
We shall always deal with differentiable distributions.

\begin{definition}
A Distribution $D$ on $M$ is integrable if it admits a unique maximal integral manifold through each point of $M$.
A Distribution $D$ is involutive if $\,[X,\,Y] \in D\,\, \text{for all } \, X,Y  \in\,D$.
\end{definition}
According to Frobenius theorem, a distribution $D$ is integrable if and only if it is involutive.

\begin{definition}
A distribution $D$ on $M$ is bracket generating if there exists an integer $r\geq1$ such that $D^r_p = T_{p}M$ for all $p\in M$, where
\begin{eqnarray*}
  D^1_p       &=& D_p, \\
  D^{s+1}_{p} &=&  D^{s}_{p}+[D_{p},D^{s}_{p}\,],\, \text{for} \,\,\,s\ge1,
\end{eqnarray*}
and  $[D_{p},D^{s}_{p}]= \{[X,Y]: X \in D_{p},Y \in D^s_{p} \}$.\\
The smallest integer $r$ such that $D^r_p = T_{p}M$ is said to be the step of the distribution $D$. If $r$ does not exist, we say that the distribution is of infinite step.
\end{definition}

\begin{definition}
A sub-Riemannian metric on a distribution $D$ is a map $g$ that assigns to each point $p\in M$ a positive definite inner product  $g_{p}:D_{p}\times D_{p}\longrightarrow \mathbb{R}$ and $g_p$ varies smoothly.
\end{definition}

\begin{definition}
A sub-Riemannian structure on $M$ is a pair $(D, g)$ where $D$ is a non-integrable (bracket generating) distribution on $M$ and $g$ is a smooth sub-Riemannian metric on $D$. In this case, $M$ is said to be a sub-Riemannian manifold.
\end{definition}

In the following we shall deal with a certain type of distributions, namely, parallelizable distributions.
 Throughout the paper, $M$ will denote an $n$-dimensional smooth manifold,\, $C^{\infty}(M)$ the algebra of smooth functions on $M$,  $\mathfrak{X}(M) \text{ the } C^{\infty}(M)$-module of smooth vector fields on $M$ and $\Gamma(D)$ the $C^{\infty}(M)$-module of smooth sections of a distribution $D$. Einstein summation convention will be applied to all pairs of repeated indices even when they are both down. We shall follow the notations and use the results of \cite{Waleed}.

\begin{definition}
A distribution $HM$ on $M$ of rank $k<n$ is said to be parallelizable if it admits $k$ independent global smooth sections $\,\undersym{X}{1}, \cdots, \undersym{X}{k}\in \Gamma(HM)$, called paralellization sections. Consequently,  $HM=\text{span}\{\,\undersym{X}{i}\,|\,i = 1,\cdots,k\}$.
A parallelizable distribution generated by $\undersym{X}{i}$ will be denoted by $(M,\, HM,\,\undersym{X}{i})$
or, simply, by $(HM,\, \undersym{X}{i})$.
\end{definition}

\begin{proposition}
Every parallelizable distribution admits a sub-Riemannian metric. \end{proposition}

\begin{proof}
Let $(HM,\;\undersym{X}{i})$ be a parallelizable distribution of rank $k<n$ on $M$. Define the $k$ differential $1$-forms $\;\undersym{\Omega}{i}:\Gamma(HM)\longrightarrow C^{\infty}(M)$ by
\begin{equation*} \label{1form}
\;\undersym{\Omega}{i}(\;\undersym{X}{j})=\delta_{ij}
\end{equation*}
We call $\;\undersym{\Omega}{i}$ the parallelization forms.
Clearly, if $ Y=Y^{i}\;\undersym{X}{i} \in \Gamma(HM)$, then
\begin{equation} \label{base}
\undersym{\Omega}{i}(Y)=Y^{i},\qquad\;\undersym{\Omega}{i}(Y)\;\undersym{X}{i}=Y.
\end{equation}
The parallelization forms $\;\undersym{\Omega}{i}$ are independent in the $C^{\infty}(M)\text{-module } \Gamma^*(HM)$.
It is then easy to show that
\begin{equation} \label{metric}
g:=\;\undersym{\Omega}{i}\otimes\;\undersym{\Omega}{i}
\end{equation}
defines a sub-Riemannian metric on $HM$.
\end{proof}

\begin{corollary}
A non-integrable parallelizable distribution $(HM,\;\undersym{X}{i})$  defines a sub-Riemannian structure on $M$ with sub-Riemannian metric $g$ (defined by $(\ref{metric}))$.
\end{corollary}

It should be noted that the parallelization sections $\;\undersym{X}{i}$ are $g$-orthonormal:
  $g(\;\undersym{X}{i},\,\;\undersym{X}{j})=\delta_{ij}$.
  Moreover, we have $g(\;\undersym{X}{i},\, Y)=\;\undersym{\Omega}{i}(Y)$ for all $Y\in \Gamma(HM)$.\\

According to \cite{Strichartz, Strichartz C}, there exists a metric extension $G$ for $g$ that makes the split $TM = HM \oplus VM\, \,\,G$-orthogonal, where $VM:=(HM)^\perp$.
This decomposition of $TM$ induces two projectors $h:TM\rightarrow HM$ and $v:TM\rightarrow VM$, called horizontal and vertical projectors, respectively. The projectors $h \text{ and } v \text{ are } C^{\infty}(M)$-linear with the properties  $h^2 =h ,\,\, v^2=v, \,\,\,h\circ v= v \circ h=0 \text{ and } h+v=id_{TM}$.
\end{mysect}
\vspace{0.5cm}

%%%%%%%%%%%%%%%%%%%%%%%%%%%%%%%%%%%%%%%%%%%%%%%%%%%%% sec. 3 %%%%%%%%%%%%%%%%%%%%%%%%%%%%%%%%%%%%%%%%%%%%%%%%%%%%%%%%%%%%%%%%

\begin{mysect}{Linear connections on parallelizable distributions}
In this section, we explore the natural sub-Riemannian structure associated with a parallelizable distribution $(HM,\;\undersym{X}{i})$. We introduce and investigate two remarkable connections on $HM$, namely, the Weitzenb\"{o}ck connection and the sub-Riemannian connection. The last one generalizes the Levi-Civita connection to the sub-Riemannian case. We shall continue to follow the notations and use the results of \cite{Waleed}.

\begin{theorem}
Let $(HM,\;\undersym{X}{i})$ be a parallelizable distribution of rank $k$ on $M$. Then, there exists a unique linear connection $\nabla$ on $HM$ for which the parallelization sections $\;\undersym{X}{i}$ are parallel:
\begin{equation}\label{AP-condition}
\nabla_{Y}\;\undersym{X}{i}=0 \,\,\,\,\,\, \forall \,Y \in\mathfrak{X}(M).
\end{equation}
\end{theorem}

\begin{proof}
To prove the uniqueness assume that $\nabla$ is a linear connection satisfying the condition $\nabla\;\undersym{X}{i}=0$. For all $Y\in\mathfrak{X}(M),\, Z\in\Gamma(HM)$ we have, by (\ref{base}) and (\ref{AP-condition}),
$$\nabla_YZ=\nabla_Y\big(\;\undersym{\Omega}{i}(Z)\;\undersym{X}{i}\big)=\;\undersym{\Omega}{i}(Z)\nabla_Y\;\undersym{X}{i}
+\big(Y\cdot\;\undersym{\Omega}{i}(Z)\big)\;\undersym{X}{i} =\big(Y\cdot\;\undersym{\Omega}{i}(Z)\big) \;\undersym{X}{i}.$$
Hence, the connection $\nabla$ is uniquely determined by the relation
\begin{equation} \label{canonical}
\nabla_YZ=\big(Y\cdot\;\undersym{\Omega}{i}(Z)\big)\;\undersym{X}{i}.
\end{equation}

To prove the existence, let $\nabla:\mathfrak{X}(M)\times\Gamma(HM)\longrightarrow\Gamma(HM)$ be defined by (\ref{canonical}).
It is easy to show that $\nabla$ is a linear connection on $HM$ with the required property.
\end{proof}
The unique linear connection $\nabla$ on $HM$ defined by (\ref{canonical}) will be called the Weitzenb\"{o}ck connection of $(HM,\;\undersym{X}{i})$.
\begin{corollary}
The Weitzenb\"ock connection $\nabla \text{ of } (HM,\;\undersym{X}{i}) \text{ is metric: } \nabla g=0$.
\end{corollary}
\begin{definition}
Let $(HM,\;\undersym{X}{i})$ be a parallelizable distribution. Let $\nabla$ be the Weitzenb\"ock connection on $(HM,\;\undersym{X}{i})$.
\begin{description}
   \item[(a)] The torsion tensor of the Weitzenb\"ock connection is defined by:
   $$T:\mathfrak{X}(M)\times\mathfrak{X}(M)\longrightarrow\Gamma(HM);$$
  \begin{equation*}\label{torsin}T(Y,\,Z)=\nabla_Y{hZ}-\nabla_Z{hY}-h[Y,\,Z].\end{equation*}
  \item[(b)] The curvature tensor of the Weitzenb\"ock connection is defined by: $$R:\mathfrak{X}(M)\times \mathfrak{X}(M)\times \Gamma(HM)\longrightarrow\Gamma(HM);$$
  \begin{equation*} \label{curvature}R(Y,\,Z)W:=\nabla_Y \nabla_ZW - \nabla_Z \nabla_YW - \nabla_{[Y,\,Z]}W.
  \end{equation*}
\end{description}
The torsion and curvature tensors of an arbitrary linear connection on $(HM,\;\undersym{X}{i})$ can be defined similarly.
\end{definition}

The torsion tensor $T$ of $\nabla$ has the properties:
$$T(\;\undersym{X}{i},\,\;\undersym{X}{j})= -h[\;\undersym{X}{i},\,\;\undersym{X}{j}],\,\,\,\,\,\,\, T(vY,\;\undersym{X}{j})=-h[vY,\;\undersym{X}{j}],\,\,\,\,\,\,\,T(vY,\,vZ)=-h[vY,\,vZ].$$
Because of the property (\ref{AP-condition}), the curvature tensor of the Weitzenb\"ock connection vanishes identically.

\begin{definition}
Given a vector field $\,W \in \mathfrak{X}(M),$ the horizontal Lie derivative with respect to $W$ of the metric tensor $g$ is defined, for all $\,Y,\,Z \in \Gamma(HM)$, by:
$$(\mathfrak{L}_{W} g)(Y,Z):=W.g(Y,Z)+  g(h[Y,\,W] ,Z) + g(h[Z,\,W] ,Y).$$

\end{definition}
It is clear that
$$(\mathfrak{L}_{fW} g)(Y,Z)=f(\mathfrak{L}_{W} g)(Y,Z)+(Y.f)\,g(hW,Z)+(Z.f)\,g(hW,Y).$$
  \begin{theorem}
On any parallelizable distribution $(M,\,HM,\undersym{X}{k})$, there exists a unique linear connection $\oversetc{\b}$, called the sub-Riemannian connection (sR-connection),  such that
 \begin{description}
   \item[(a)] $\oversetc{\b}$ is metric: $(\oversetc{\b}_{W} g)(Y,\,Z)=0\,\, \forall \, Y,\,Z \in \Gamma(HM),\,W \in \mathfrak{X}(M).$
   \item[(b)] $\oversetc{T}(HM,\,HM)=0.$
   \item[(c)] $g(\oversetc{T}(V,\,Y), Z)=g(\oversetc{T}(V,\,Z), Y),\,\, \forall \, Y,\,Z \in \Gamma(HM),\,V \in \Gamma(VM).$
 \end{description}
 \end{theorem}
\begin{proof}
Let $W \in \mathfrak{X}(M)$ and \,$Z\in \Gamma(HM)$. Set
\begin{equation}\label{SR connection}
\oversetc{\nabla}_WZ=\widehat{\nabla}_WZ -\frac{1}{2}(\mathfrak{L}_{\;\undersym{X}{i}} g)(hW,Z)\, \;\undersym{X}{i}-\frac{1}{2}g(Z,\,h[vW,\,\undersym{X}{k}])\undersym{X}{k},
\end{equation}
where $ \widehat{\nabla}_YZ:=\frac{1}{2}(\nabla_YZ+\nabla_{Z}hY+h[Y,Z])$.

It is clear that, $\oversetc{\nabla}$ is a linear connection with the desired properties. For example, let us prove the property (c):
\begin{eqnarray*}\oversetc{T}(V,\;\undersym{X}{k})&=&\oversetc{\b}_{V}\;\undersym{X}{k}-h[V,\;\undersym{X}{k}]=
\widehat{\nabla}_{V} \;\undersym{X}{k}-\frac{1}{2}g(\;\undersym{X}{k},\,h[V,\,\undersym{X}{i}])\undersym{X}{i}-h[V,\;\undersym{X}{k}] \\
&=&\frac{1}{2}\Big\{\nabla_{V} \;\undersym{X}{k}+\nabla_{\;\undersym{X}{k}} hV+h[V,\;\undersym{X}{k}]-
g(\;\undersym{X}{k},\,h[V,\,\undersym{X}{i}])\undersym{X}{i} -2h[V,\;\undersym{X}{k}]\Big\} \\
&=&-\frac{1}{2}\Big\{g(\;\undersym{X}{k},\,h[V,\,\undersym{X}{i}])\undersym{X}{i} +h[V,\;\undersym{X}{k}]\Big\}.\end{eqnarray*}
Hence,
\begin{eqnarray*}
  g(\oversetc{T}(V,\;\undersym{X}{k}),\;\undersym{X}{j}) &=&- \frac{1}{2}\Big\{g(\;\undersym{X}{k},\,h[V,\,\undersym{X}{i}])g(\undersym{X}{i},\;\undersym{X}{j}) +g(h[V,\;\undersym{X}{k}],\;\undersym{X}{j})\Big\}\\
  &=&-\frac{1}{2}\Big\{g(\;\undersym{X}{k},\,h[V,\,\undersym{X}{j}]) +g(h[V,\;\undersym{X}{k}],\;\undersym{X}{j})\Big\} \\
  &=& g(\oversetc{T}(V,\;\undersym{X}{j}),\;\undersym{X}{k}).
\end{eqnarray*}

\par For the uniqueness, assume that $\overline{\nabla}$ is another linear connection satisfying (a), (b) and (c).
Define the tensor $\psi(W,\,Z):= \oversetc{\nabla}_WZ -\overline{\nabla}_WZ$. Then for $Y,\,Z \in \Gamma(HM),\,V \in \Gamma(VM)$,
\begin{eqnarray*}
  g(\psi(V,\,Z),\,Y)&\overset{\;\text{(a)}}{=}& -g(Z,\,\psi(V,\,Y))=-g(Z,\,\oversetc{T}(V,\,Y)-\overline{T}(V,\,Y)) \\
  &\overset{\;\text{(c)}}{=}& -g(Y,\,\oversetc{T}(V,\,Z)-\overline{T}(V,\,Z))\\
   &=& -g(Y,\,\psi(V,\,Z)).
\end{eqnarray*}
On the other hand, for $X,\,Y,\,Z \in \Gamma(HM),$
\begin{eqnarray*}
  g(\psi(X,\,Z),\,Y)&\overset{\;\text{(a)}}{=}& -g(Z,\,\psi(X,\,Y))\overset{\;\text{(b)}}{=}-g(Z,\,(Y,\,X))
  \overset{\;\text{(a)}}{=} g(X,\,\psi(Y,\,Z))\\
  &\overset{\;\text{(b)}}{=}& g(X,\,\psi(Z,\,Y))\overset{\;\text{(a)}}{=}-g(Y,\,\psi(Z,\,X))
  \overset{\;\text{(b)}}{=}-g(Y,\,\psi(X,\,Z)).\end{eqnarray*}
Hence, $\psi(W,\,Z):= 0\,\,\, \forall \,W \in \mathfrak{X}(M),\,Z\in \Gamma(HM)$, which completes the proof.
\end{proof}
%The unique linear connection $\oversetc{\nabla}$ on $HM$ defined by (\ref{SR connection}) will be called %the sub-Riemannian connection (sR-connection).

 \begin{remark}
 {\em
\textbf{(a)} The non vanishing counterparts of the torsion tensor $\oversetc{T}$ of\, $\oversetc{\nabla}$ are given, for all $V,U \in \Gamma(VM)$, by
   \begin{eqnarray*}
    \oversetc{T}(\;\undersym{X}{k},\,V)&=& \frac{1}{2}\Big\{T(\;\undersym{X}{k},\,V)- g(h[\;\undersym{X}{k},\,V],\;\undersym{X}{i})\;\undersym{X}{i}\Big\},\\
    \oversetc{T}(V,\,U)&=&-h[V,\,U].
   \end{eqnarray*}
   On the other hand, $\oversetc{T}$ vanishes on the horizontal distribution. The connection $\oversetc{\nabla}$ is thus a generalization of the Levi-Civita connection to the sub-Riemannian case. The advantage of formula (\ref{SR connection}) is that it gives the connection $\oversetc{\b}$ an explicit form, contrary to the Levi-Civita connection.\vspace{4pt}\\
 \textbf{(b)} If, in particular, $M$ is parallelizable ($k=n$), the sR-connection is just the well known Levi-Civita connection of the parallelizable manifold $M$ \cite{Waleed, AMR}.
 }
 \end{remark}
 \end{mysect}

%%%%%%%%%%%%%%%%%%%%%%%%%%%%%%%%%%%%%%%%%%%%%%%%%%%%% sec. 4 %%%%%%%%%%%%%%%%%%%%%%%%%%%%%%%%%%%%%%%%%%%%%%%%%%%%%%%%%%%%%%%%

\begin{mysect}{The sphere $S^3$}
Let $M= S^3$, the 3-sphere, and let $(y_0,y_1,y_2,y_3)$ be the coordinates on $S^3$. Consider the parallelization vector fields on $S^3$ given by \cite{parallelizable spheres}:
\begin{eqnarray*}
  \undersym{X}{1} &=& -y_{2}\,\partial_{y_{0}}+y_{3}\,\partial_{y_{1}}+y_{0}\,\partial_{y_{2}}-y_{1}\,\partial_{y_{3}}, \\
  \undersym{X}{2} &=& -y_{3}\,\partial_{y_{0}}-y_{2}\,\partial_{y_{1}}+y_{1}\,\partial_{y_{2}}+y_{0}\,\partial_{y_{3}}, \\
  \undersym{X}{3} &=& -y_{1}\,\partial_{y_{0}}+y_{0}\,\partial_{y_{1}}-y_{3}\,\partial_{y_{2}}+y_{2}\,\partial_{y_{3}}=
(1/2)[\;\undersym{X}{1},\;\undersym{X}{2}].
\end{eqnarray*}
Let $HM= \text{span}\{\undersym{X}{1},\;\undersym{X}{2}\}$ and $VM= \text{span}\{\undersym{X}{3}\}$.
The distribution $HM$ is non-integrable and bracket generating of step 2. The parallelization forms associated with $\;\undersym{X}{1},\,\;\undersym{X}{2}$ are given by:
$$\;\undersym{\Omega}{1}=-y_{2}\,dy_{0}+y_{3}\,dy_{1}+y_{0}\,dy_{2}-y_{1}\,dy_{3},$$
$$\;\undersym{\Omega}{2}=-y_{3}\,dy_{0}-y_{2}\,dy_{1}+y_{1}\,dy_{2}+y_{0}\,dy_{3}.$$
The sub-Riemannian metric of $HM$, defined by (\ref{metric}), is given by
\begin{eqnarray*}
  g &=& (y_{2}^{2}+ y_{3}^{2})\,({dy_{0}}^{2} +{dy_{1}}^{2}) +
   (y_{1}^{2}+ y_{0}^{2})\,({dy_{2}}^{2} +{dy_{3}}^{2}) \\
   & & +2(y_{0}\,y_{3}-y_{1}\,y_{2})(dy_{0}\,dy_{3}-dy_{1}\,dy_{2} )-  2(y_{1}\,y_{3}-y_{0}\,y_{2})(dy_{0}\,dy_{2}-dy_{1}\,dy_{3} ).
\end{eqnarray*}
%$$g:=\;\undersym{\Omega}{1}\otimes\;\undersym{\Omega}{1}+\;\undersym{\Omega}{2}\otimes\;\undersym{\Omega}{2}.$$
We have $g(HM,HM)=\langle HM,HM\rangle$ and  $\langle HM,VM\rangle=0$,  where $\langle.\,,\,.\rangle$ is the usual inner product of\, $\mathbb{R}^4$.
\begin{itemize}
  \item  The Weitzenb\"ock connection $\nabla$ defined by (\ref{canonical}) has coefficients
      $$\nabla_{Y} \;\undersym{X}{1} = \nabla_{Y} \;\undersym{X}{2} = 0, \text{ where } \,\, Y=\undersym{X}{1},\, \undersym{X}{2} \text{ or } \undersym{X}{3}.$$
\end{itemize}
The torsion tensor $T$ of $\nabla$ is given by
$$T(\;\undersym{X}{1},\;\undersym{X}{2})=0, \,\,\,T(\;\undersym{X}{1},\,\undersym{X}{3})=-2\;\undersym{X}{2},\,\,\,T(\;\undersym{X}{2},\,\undersym{X}{3})=2\;\undersym{X}{1}$$
and the curvature tensor of $\nabla$ vanish identically.

 \begin{itemize}
  \item The sR-connection defined by (\ref{SR connection}) has  coefficients
\end{itemize}
$$\oversetc{\nabla}_{\;\undersym{X}{i}} \;\undersym{X}{j}=0 \,\, (i,j\in \{1,2\}),\,\,\,\,\oversetc{\nabla}_{\undersym{X}{3}} \;\undersym{X}{1}=-3\;\undersym{X}{2},\,\,\,\,\,\,\oversetc{\nabla}_{\undersym{X}{3}} \;\undersym{X}{2}=3\;\undersym{X}{1}.$$
The non vanishing components of the torsion tensor $\oversetc{T}$ of $\oversetc{\nabla}$ are
$$\oversetc{T}(\;\undersym{X}{1},\,\undersym{X}{3})= \;\undersym{X}{2},\,\,\,\, \oversetc{T}(\;\undersym{X}{2},\,\undersym{X}{3})=-\;\undersym{X}{1}$$
and the non vanishing components of the curvature tensor $\oversetc{R}$ of $\oversetc{\nabla}$ are
$$\oversetc{R}(\;\undersym{X}{1},\;\undersym{X}{2})\;\undersym{X}{1}=6\;\undersym{X}{2},\,\,\,\,\,\oversetc{R}(\;\undersym{X}{1},\;\undersym{X}{2})\;\undersym{X}{2} =-6\;\undersym{X}{1}.$$

\vspace{0pt}
\begin{remark}{\em
It should be noted that $\undersym{X}{2} =\frac{1}{2}[\,\undersym{X}{1}, \undersym{X}{3}],\,\,\undersym{X}{1}=\frac{1}{2}[\,\undersym{X}{2}, \undersym{X}{3}]$. This implies that the distributions span$\{\undersym{X}{1}, \undersym{X}{3}\} \text{ and span}\{\undersym{X}{2}, \undersym{X}{3}\}$ are also bracket generating of step 2 and we can perform the same calculation as above for each of them. Consequently, there are exactly three distinct sub-Riemannian parallelizable structures on $S^3$.
}
\end{remark}
\end{mysect}

%%%%%%%%%%%%%%%%%%%%%%%%%%%%%%%%%%%%%%%%%%%%%%%%%%%%% sec. 5 %%%%%%%%%%%%%%%%%%%%%%%%%%%%%%%%%%%%%%%%%%%%%%%%%%%%%%%%%%%%%%%%

\begin{mysect}{The sphere $S^7$}

Consider the 7-sphere $S^7$. For each point $p\in S^7$, the vector fields $Y_1,\cdots,Y_7$ form an orthonormal frame of $T_pS^7$, where \cite{parallelizable spheres}

$$
\undersym{X}{1}=-y_2\partial_{y_0}+y_3\partial_{y_1}+y_0\partial_{y_2}-
y_1\partial_{y_3}-y_6\partial_{y_4}+y_7\partial_{y_5}+
y_4\partial_{y_6}-y_5\partial_{y_7}
$$ $$
\undersym{X}{2}=-y_3\partial_{y_0}-y_2\partial_{y_1}+y_1\partial_{y_2}+
y_0\partial_{y_3}+y_7\partial_{y_4}+y_6\partial_{y_5}-
y_5\partial_{y_6}-y_4\partial_{y_7}
$$ $$
\undersym{X}{3}=-y_4\partial_{y_0}+y_5\partial_{y_1}+y_6\partial_{y_2}-
y_7\partial_{y_3}+y_0\partial_{y_4}-y_1\partial_{y_5}-
y_2\partial_{y_6}+y_3\partial_{y_7}
$$ $$
\undersym{X}{4}=-y_5\partial_{y_0}-y_4\partial_{y_1}-y_7\partial_{y_2}-
y_6\partial_{y_3}+y_1\partial_{y_4}+y_0\partial_{y_5}+
y_3\partial_{y_6}+y_2\partial_{y_7}
$$ $$
\undersym{X}{5}=-y_6\partial_{y_0}+y_7\partial_{y_1}-y_4\partial_{y_2}+
y_5\partial_{y_3}+y_2\partial_{y_4}-y_3\partial_{y_5}+
y_0\partial_{y_6}-y_1\partial_{y_7}
$$ $$
\undersym{X}{6}=-y_7\partial_{y_0}-y_6\partial_{y_1}+y_5\partial_{y_2}+
y_4\partial_{y_3}-y_3\partial_{y_4}-y_2\partial_{y_5}+
y_1\partial_{y_6}+y_0\partial_{y_7}
$$ $$
\undersym{X}{7}=-y_1\partial_{y_0}+y_0\partial_{y_1}-y_3\partial_{y_2}+
y_2\partial_{y_3}-y_5\partial_{y_4}+y_4\partial_{y_5}-
y_7\partial_{y_6}+y_6\partial_{y_7}$$
Let $HM= \text{Span}\{\undersym{X}{1}, ...,\undersym{X}{6}\}$ and $VM= \text{Span}\{\undersym{X}{7}\}$. Clearly, the distribution $HM$ is non-integrable and bracket generating of step 2. The parallelization forms associated with $\{\undersym{X}{1}, ...,\undersym{X}{6}\}$ are given by
$$
\;\undersym{\Omega}{1}=-y_2\, d{y_0}+y_3 \,d{y_1}+y_0\, d{y_2}-
y_1\, d{y_3}-y_6\, d{y_4}+y_7\, d{y_5}+
y_4 \,d{y_6}-y_5\, d{y_7}
$$ $$
\;\undersym{\Omega}{2}=-y_3\, d{y_0}-y_2\, d{y_1}+y_1\, d{y_2}+
y_0\, d{y_3}+y_7\, d{y_4}+y_6\, d{y_5}-
y_5\, d{y_6}-y_4\, d{y_7}
$$ $$
\;\undersym{\Omega}{3}=-y_4 \,d{y_0}+y_5\, d{y_1}+y_6 \,d{y_2}-
y_7\, d{y_3}+y_0\, d{y_4}-y_1\, d{y_5}-
y_2\, d{y_6}+y_3\, d{y_7}
$$ $$
\;\undersym{\Omega}{4}=-y_5\,d{y_0}-y_4\,d{y_1}-y_7\,d{y_2}-
y_6\,d {y_3}+y_1\, d{y_4}+y_0\, d{y_5}+
y_3\, d{y_6}+y_2 \,d{y_7}
$$ $$
\;\undersym{\Omega}{5}=-y_6\,d{y_0}+y_7\,d{y_1}-y_4\,d{y_2}+
y_5\,d{y_3}+y_2\,d{y_4}-y_3\,d{y_5}+
y_0\,d{y_6}-y_1\,d{y_7}
$$ $$
\;\undersym{\Omega}{6}=-y_7\,d{y_0}-y_6\, d{y_1}+y_5\, d{y_2}+
y_4\, d{y_3}-y_3\,d{y_4}-y_2\, d{y_5}+
y_1\, d{y_6}+y_0\, d{y_7}.
$$
The metric $g$ of $HM$, defined by (\ref{metric}), is given by %$$g:=\;\undersym{\Omega}{2}\otimes\;\undersym{\Omega}{2}+...+\;\undersym{\Omega}{7}\otimes\;\undersym{\Omega}{7}.$$
\begin{eqnarray*}
  g &=& (1-{y_0}^{2}-{y_1}^2 )({d{y_0}}^2 + {d{y_1}}^2) + (1-{y_2}^{2}-{y_3}^{2})(d{y_2}^2+ {d{y_3}}^2)  \\
   & & +(1- {y_4}^{2}-{y_5}^{2}) (d{y_4}^2  +d{y_5}^2) +(1-{y_6}^{2}-{y_7}^{2})(d{y_6}^2 +d{y_7}^2) \\
  &  & +2(-y_1 \,y_6 + y_7\, y_0)(d{y_6}\,d{y_1}-d{y_7}\,d{y_0} )
  +2 (-y_3\, y_0 + y_2\, y_1)(d{y_3}\,d{y_0}- d{y_1}\,d{y_2}) \\
  & & +2(-y_0 \,y_5 + y_1 \,y_4)(d{y_0}\,d{y_5}- d{y_4}\,d{y_1} )
  -2(y_7\, y_1 + y_6\, y_0)(d{y_6}\,d{y_0}+d{y_7}\,d{y_1}) \\
  & & -2(y_1\, y_3 + y_2\, y_0)(d{y_3}\,d{y_1}- d{y_0}\,d{y_2})
  - 2(y_7\, y_5 + y_6\, y_4)(d{y_5}\,d{y_7}+d{y_4}\,d{y_6})\\
  & &  -2(y_7\, y_3 + y_2 \,y_6)(d{y_2}\,d{y_6}-d{y_3}\,d{y_7})+
   2(-y_3\, y_4 + y_5\, y_2)(d{y_4}\,d{y_3}-d{y_2}\,d{y_5}) \\
  & & + 2(-y_6\, y_5 + y_7 \,y_4)(d{y_6}\,d{y_5} -d{y_4}\,d{y_7})
  + 2(-y_2\, y_7 + y_3\, y_6)(d{y_2}\,d{y_7} - d{y_3}\,d{y_6})\\
  & & -2(y_3\, y_5 + y_4\, y_2)(d{y_2}\,d{y_4}+d{y_3}\,d{y_5}) -2(y_1\, y_5 + y_0\, y_4)(d{y_4}\,d{y_0}+d{y_3}\,d{y_5})
\end{eqnarray*}
We have $g(HM,\,HM)=\langle HM,\,HM\rangle, \text{  and  } \langle HM,\,VM\rangle=0,$ where $\langle.\,,\,.\rangle$ is the usual inner product on $\mathbb{R}^8$.

\begin{itemize}
  \item  The Weitzenb\"ock connection $\nabla$ defined by (\ref{canonical}) has coefficients
      $$\nabla_{\undersym{X}{i}} \;\undersym{X}{j} = 0,\,\,\text{for } 1\leq i \leq 7, \, 1\leq j \leq 6.$$
\end{itemize}
As an illustration, some components of the torsion tensor $T$ of $\nabla$ are given by
$$
T(\;\undersym{X}{1},\;\undersym{X}{7})=  2(y_3 \partial_{y_0}+y_2\partial_{y_1}-y_1\partial_{y_2}-
y_0\partial_{y_3}+y_7\partial_{y_4}+y_6\partial_{y_5}-
y_5\partial_{y_6}-y_4\partial_{y_7}),$$
$$T(\;\undersym{X}{6},\;\undersym{X}{7})=2(-y_6\partial_{y_0}+y_7\partial_{y_1}+y_4\partial_{y_2}-
y_5\partial_{y_3}-y_2\partial_{y_4}+y_3\partial_{y_5}+
y_0\partial_{y_6}-y_1\partial_{y_7})
$$
and the curvature tensor of $\nabla$ vanishes identically.

\begin{itemize}
  \item The sR-connection defined by (\ref{SR connection}) has the properties:
\end{itemize}
$$\oversetc{\nabla}_{\undersym{X}{i}} \;\undersym{X}{j} = -\oversetc{\nabla}_{\undersym{X}{j}} \;\undersym{X}{i},\,\, \,\,
\,\oversetc{\nabla}_{\undersym{X}{7}} \;\undersym{X}{i} = \frac{3}{2}\,T(\;\undersym{X}{i},\;\undersym{X}{7}),\,\,\, 1\leq i,j \leq 6.$$
For the coefficients of $\oversetc{\nabla}$, we have, for example,\begin{eqnarray*}
\oversetc{\nabla}_{\undersym{X}{1}} \;\undersym{X}{2}=-\oversetc{\nabla}_{\undersym{X}{2}} \;\undersym{X}{1} &=& 2 ( y_7^2  +y_6^2 +y_5^2 +y_4^2)(y_1 \,dy_{0} - y_0 \,dy_{1} +y_3 \,dy_{2}-y_2 \,dy_{3}) \\
              & & +2 ( y_0^2  +y_1^2 +y_2^2 +y_3^2)(y_5 \,dy_{4} - y_4 \,dy_{5} +y_7 \,dy_{6}-y_6 \,dy_{7}).
            \end{eqnarray*}
The non vanishing components of the torsion tensor $\oversetc{T}$ of $\oversetc{\nabla}$ are
$$\oversetc{T}(\;\undersym{X}{i},\;\undersym{X}{7})=-\frac{1}{2}\;T(\;\undersym{X}{i},\;\undersym{X}{7}),
\,\,\,1\leq i \leq 6.$$

It is to be noted that the other non-vanishing components of $T$ and the non-vanishing components of $\oversetc{R}$ have not been written down. They have been computed using Maple program but they are so long. The following result is also proved using Maple.
\begin{proposition}
The parallelization sections $\undersym{X}{i},\,\, i=1,\ldots,6,$ are Killing sections: ${\mathcal{L}}_{\undersym{X}{i}}g=0$.
\end{proposition}
The next table gives the different parallelizable distributions (PD) defined on $S^7$, where we use the notation: $X_{kl}:=[{\undersym{X}{k}}, {\undersym{X}{l}}]$.

\begin{center}
\small{\begin{tabular} {|c|c|c|c|c|}\hline
\multirow{2}{*}{$HM$ }&\multirow{2}{*}{ Rank}&\multirow{2}{*}{Independent commutators} &\multirow{2}{*}{ HM is } &\multirow{2}{*}{Step}
\\[0.2 cm]spanned by & of HM&&bracket generating&
\\[0.2cm]
\hline
\multirow{2}{*}{$\undersym{X}{1}, ...,\undersym{X}{6}$}&\multirow{2}{*}{6}&\multirow{2}{*}{$X_{12}, X_{13},X_{14},X_{15},X_{16},X_{23},X_{24}, X_{25}, $}&\multirow{2}{*}{Yes}&\multirow{2}{*}{2}\\[0.3cm]&&$X_{26},X_{34},X_{35},X_{36},X_{45},X_{46},X_{56}$&&  \\[0.2cm]
\hline
\multirow{2}{*}{$\undersym{X}{1}, ...,\undersym{X}{5}$}&\multirow{2}{*}{5}&\multirow{2}{*}{$(X_{12}, X_{13}),(X_{12}, X_{23}),(X_{12}, X_{14}),$}& \multirow{2}{*}{Yes}&\multirow{2}{*}{2}\\[0.3 cm] &&$(X_{12}, X_{24}),(X_{12}, X_{34}),(X_{12}, X_{15}),$ &&
\\[0.1 cm]&&$(X_{12}, X_{25}),(X_{12}, X_{35}),(X_{12}, X_{45})$&&
\\[0.2 cm]\hline
\multirow{2}{*}{$\undersym{X}{1}, ...,\undersym{X}{4}$}&\multirow{2}{*}{4}&\multirow{2}{*}{$(X_{12}, X_{13}, X_{23}),(X_{12}, X_{13}, X_{14}),$}& \multirow{2}{*}{Yes}&\multirow{2}{*}{2}\\[0.3 cm]&&$(X_{12}, X_{13}, X_{34}),(X_{12}, X_{13}, X_{24}) $&&
\\[0.2 cm]\hline
\multirow{2}{*}{$\undersym{X}{1}, \undersym{X}{2},\undersym{X}{3}$}&\multirow{2}{*}{3}&\multirow{2}{*}{$(X_{12}, X_{13}, X_{23})$}& \multirow{2}{*}{No}&\multirow{2}{*}{infinite}
\\[0.3 cm]\hline
\multirow{2}{*}{$\undersym{X}{1},\undersym{X}{2}$}&\multirow{2}{*}{2}&\multirow{2}{*}{$X_{12}$}& \multirow{2}{*}{No}&\multirow{2}{*}{infinite}
\\[0.3 cm]\hline
\end{tabular}}
\end{center}

\vspace{5pt}
The above table provides some sort of classification of sub-Riemannian parallelizable distributions on $S^7$.
We have 7 PD's of rank 6 (spanned by different choices of 6 sections from the 7 ones $\undersym{X}{1}, ...,\undersym{X}{7}$). Similarly, there are 21 PD's of rank 5, 35 PD's of rank 4, 35 PD's of rank 3 and 21 PD's of rank 2.

For example, the second row of this table concerns with the PD of rank 6 spanned by $\undersym{X}{1}, ...,\undersym{X}{6}$. It is non-integrable and bracket generating (BG) of step 2. The third column (intersecting the second row) gives the commutators independent with  $\undersym{X}{1}, ...,\undersym{X}{6}$. That is,  $\undersym{X}{1}, ...,\undersym{X}{6}$ together with $X_{12}$ are independent and $\undersym{X}{1}, ...,\undersym{X}{6}$ together with $X_{13}$ are independent, ..., etc. It should be noted that besides the above mentioned 7 PD's of rank 6, there are many other PD's of rank 6: we may take $\{\undersym{X}{1}, ...,\undersym{X}{5}, X_{12}\}$, $\{\undersym{X}{1}, ...,\undersym{X}{5}, X_{13}\}$, $\{\undersym{X}{1}, ...,\undersym{X}{4}, X_{12}, X_{13}\}$, ..., etc.

The same discussion can be made for the other rows of the table where we consider PD's of rank 5, 4, 3, 2 on $S^7$. This gives many non-integrable PD's which are either BG or non BG.
\end{mysect}

%%%%%%%%%%%%%%%%%%%%%%%%%%%%%%%%%%%%%%%%%%%%%%%%%%%%%%%%%%%%%%% References %%%%%%%%%%%%%%%%%%%%%%%%%%%%%%%%%%%%%%%%%%%%%%%%%%%%%%%%%%%

\bibliographystyle{plain}

\end{document}